\documentclass[11pt]{article}%
\usepackage{amssymb}
\usepackage{amsmath}
\usepackage{amsfonts}
\usepackage{graphicx}
\usepackage{color}
\usepackage[top=2cm,bottom=2cm,left=2.5cm,right=2.5cm]{geometry}%
\providecommand{\U}[1]{\protect\rule{.1in}{.1in}}
\def\R{\mathbb R}
\def\rN{{{\R}^N}}

\newtheorem{theorem}{Theorem}[section]
\newtheorem{lemma}{Lemma}[section]

\newtheorem{corollary}[theorem]{Corollary}

\newtheorem{definition}[theorem]{Definition}

\newenvironment{proof}[1][Proof]{\textbf{#1.} }{\ \rule{0.5em}{0.5em}}
\begin{document}

\title{ Comparison results for nonlinear anisotropic parabolic problems}
\author{A. Alberico\thanks{Istituto per le Applicazioni del Calcolo ``M. Picone'',
Sez. Napoli, Consiglio Nazionale delle Ricerche (CNR), Via P. Castellino 111,
80131 Napoli, Italy. E--mail:a.alberico@na.iac.cnr.it} -- G. di
Blasio\thanks{Dipartimento di Matematica e Fisica, Seconda Universit\`{a}
degli Studi di Napoli, Via Vivaldi, 43 - 81100 Caserta, Italy. E--mail:
giuseppina.diblasio@unina2.it} -- F. Feo\thanks{Dipartimento di Ingegneria,
Universit\`{a} degli Studi di Napoli \textquotedblleft
Pathenope\textquotedblright, Centro Direzionale Isola C4 80143 Napoli, Italy.
E--mail: filomena.feo@uniparthenope.it}}
\maketitle

\begin{abstract}
{\scriptsize {\negthinspace\negthinspace\negthinspace\ } Comparison results
for solutions to the Dirichlet problems for a class of nonlinear, anisotropic
parabolic equations are established. These results are obtained through a
semi-discretization method in time after providing estimates for solutions to
anisotropic elliptic problems with zero-order terms.
%problem an integral comparison estimates for solutions
%to anisotropic elliptic problems with zero-order
%terms.
%which consists into
%approximating the solution of the relevant  parabolic problem
%with a sequence of solutions to anisotropic elliptic problems with zero-order
%terms
%after providing an integral comparison result for such elliptic
%problem.
}

\end{abstract}


%\date{}

\numberwithin{equation}{section}\numberwithin{equation}{section}

\section{Introduction}

In this work we prove comparison results for a class of nonlinear anisotropic
parabolic problems whose model case is%
\begin{equation}
\left\{
\begin{array}
[c]{lll}%
\partial_{t}u-\overset{N}{\underset{i=1}{\sum}}\left(  \alpha_{i}\left\vert
\partial_{x_{i}}u\right\vert ^{p_{i}-2}\partial_{x_{i}}u\right)  _{x_{i}%
}=f(x,t) & \quad\text{in }Q_{T}:=\Omega\times\left(  0,T\right)  & \\
u(x,0)=u_{0}(x) & \quad\text{in}\;\Omega & \\
u(x,t)=0 & \quad\text{on}\;\partial\Omega\times\left(  0,T\right)  , &
\end{array}
\right.  \label{intro}%
\end{equation}
where $\Omega$ is an open, bounded subset of $\mathbb{R}^{N}$ with Lipschitz
continuous boundary, $N\geq2$, $T>0,\alpha_{i}>0$ $\ $and $p_{i}\geq1$ for
$i=1,...,N${ such that} their harmonic mean $\bar{p}>1$ {and the data $f$}
and{ $u_{0}$ have a suitable summability.}

Problem (\ref{intro}) provides the mathematical models for natural phenomena
in biology and fluid mechanics. For example, they are the mathematical
description of the dynamics of fluids in anisotropic media when the
conductivities of the media are different in different directions. They also
appear in biology as a model for the propagation of epidemic diseases in
heterogeneous domains.

In the last years, anisotropic problems have been largely studied by many
authors (see \textit{e.g.} \cite{antontsev-chipot-08, BMS, castro2,
DiNardo-Feo-Guibe, DiNardo-Feo, FGK, FGL, FS, Gi, Mar}). The growing interest
has led to an extensive investigation also for problems governed by fully
anisotropic grows condition (see \textit{e.g. }\cite{A, AC, AdBF1, Clocal,
cianchi anisotropo}) and problems related to different type of anisotropy (see
\textit{e.g.} \cite{AFTL, BFK, DPdB,DPG}).

We emphasize that, when $p_{i}=p\neq2$ for $i=1,...,N$ the anisotropic
diffusion operator in problem (\ref{intro}) coincides with the so-called
pseudo-Laplacian operator, whereas when $p_{i}=2$ for $i=1,...,N$ it coincides
with usual Laplacian.

Symmetrization methods in a priori estimates for solutions to isotropic
parabolic problems were widely used (see \textit{e.g.} \cite{ALT1},\cite{B},
\cite{D}, \cite{FV}, \cite{MR}, \cite{V} and the bibliography starting with it).

As in the isotropic setting (see \textit{e.g.} \cite{Ta1}), if $w$ solves the
stationary anisotropic problem%
\[
\ \left\{
\begin{array}
[c]{ll}%
-\underset{i=1}{\overset{N}{%
%TCIMACRO{\dsum }%
%BeginExpansion
{\displaystyle\sum}
%EndExpansion
}}\left(  \alpha_{i}\left\vert w_{x_{i}}\right\vert ^{p_{i}-2}w_{x_{i}%
}\right)  _{x_{i}}=f\left(  x\right)  & \mbox{  in }\Omega\vspace{0.2cm}\\
& \\
w=0 & \mbox{ on }\partial\Omega,
\end{array}
\right.
\]
rearrangement methods allows to obtain a pointwise comparison result for $w$
(see \cite{cianchi anisotropo}). Namely,%
\begin{equation}
w^{\bigstar}\left(  x\right)  \!\leq\,z\left(  x\right)  \text{ \ \ \ \ \ for
\textit{a.e.} \ }\Omega^{\bigstar}\!\!\vspace{0.2cm}, \label{comparison}%
\end{equation}
where $\Omega^{\bigstar}$ is the ball centered in the origin such that
$|\Omega^{\bigstar}|=|\Omega|$, $w^{\bigstar}$ is the symmetric rearrangement
of a solution $w$ to problem (\ref{intro}) and $z$ is the radial solution to
the following isotropic problem
\begin{equation}
\!\!\left\{  \!\!%
\begin{array}
[c]{ll}%
\!\!-\operatorname{div}\left(  \Lambda\left\vert \nabla z\right\vert
^{\overline{p}-2}\nabla v\right)  \!=\!\!f^{\bigstar}\left(  x\right)  \! &
\mbox{  in  }\Omega^{\bigstar}\!\!\vspace{0.2cm}\\
& \\
\!\!z=0 & \!\mbox{ on }\partial\Omega^{\bigstar}\!\!,
\end{array}
\!\!\right.  \label{Prob simm pi}%
\end{equation}
with $\Lambda$ a suitable positive constant, $\overline{p}$ the harmonic mean
of exponents $p_{1},\dots,p_{N}$ and $f^{\bigstar}$ the symmetric decreasing
rearrangement of $f.$

In the parabolic setting, the pointwise comparison (\ref{comparison}) need not
hold, nevertheless it is possible to prove for fixed $t\in\left(  0.T\right)
,$ the following integral comparison result
\begin{equation}
\int_{0}^{s}u^{\ast}(\sigma,t)\;d\sigma\leq\int_{0}^{s}v^{\ast}(\sigma
,t)\;d\sigma\qquad\hbox{in $(0, |\Omega|)$,} \label{integral_into}%
\end{equation}
where$\ u^{\ast}$ and $v^{\ast}$\ are the decreasing rearrangement with
respect to the space variable of the solution $u$ to problem (\ref{intro}) and
of the solution $v$ to the following isotropic "symmetrized" problem
\begin{equation}
\left\{
\begin{array}
[c]{ll}%
v_{t}-\mathrm{{div}}\left(  \Lambda|\nabla v|^{\bar{p}-2}\nabla v\right)
=f^{\bigstar}(x,t) & \quad
\hbox{in $Q^{\bigstar}_T:= \Omega^{\bigstar}\times [0, T]$}\\
v(x,0)=u_{0}^{\bigstar}(x) & \quad\hbox{in $\Omega^{\bigstar}$}\\
v(x,t)=0 & \quad\hbox{on $\partial \Omega\times [0, T],$}
\end{array}
\right.  \label{symm_Pr}%
\end{equation}
respectively. We stress that in contrast to the isotropic case not only the
space domain and the data of problem (\ref{symm_Pr}) are symmetrized with
respect to the space variable, but also the ellipticity condition is subject
to an appropriate symmetrization. Indeed the diffusion operator in problem
(\ref{symm_Pr}) is the isotropic $\overline{p}-$Laplacian.

In order to obtain the integral comparison result (\ref{integral_into}) we
will use the method of semi-discretization in time. This approach was firstly
used by (\cite{V}) and (\cite{ALT1}) and consists into approximating the
solution of a parabolic problem with a sequence of solutions to elliptic
problems with zero-order terms. For this reason, we first prove an integral
comparison result for such elliptic problems and then, passing to the limit,
we obtain (\ref{integral_into}). We emphasize that integral comparison
(\ref{integral_into}) implies a priori estimates for any Lorentz norm of
$u\left(  \cdot,t\right)  $ in terms of the same norm of $v\left(
\cdot,t\right)  $ for any fixed $t>0$. Moreover, we study the asymptotic
behavior of solution $u\left(  \cdot,t\right)  $ as the time variable $t$ goes
to infinity. The paper is organized as follows. In Section 2 we recall some
backgrounds on the anisotropic spaces and on the properties of symmetrization.
In Section 3 we prove a integral comparison result for elliptic anisotropic
problems and the main results.

\section{Preliminaries}

\subsection{Anisotropic spaces}

Let $\Omega$ be an open, bounded subset of $\mathbb{R}^{N}$ with Lipschitz
continuous boundary, $N\geq2$, and let $1\leq p_{1},\ldots,p_{N}<\infty$ be
$N$ real numbers. We define the anisotropic Sobolev space $W_{0}^{1,p_{i}%
}(\Omega)$ as the closure of $C_{0}^{\infty}(\Omega)$ with respect to the
norm
\[
\left\Vert u\right\Vert _{W_{0}^{1,p_{i}}(\Omega)}=\left\Vert u\right\Vert
_{L^{1}(\Omega)}+\left\Vert \partial_{x_{i}}u\right\Vert _{L^{p_{i}}(\Omega
)}.
\]
In this anisotropic setting, a \textit{Poincar\'{e}-type inequality} holds
(see \cite{FGK}). If $u\in W_{0}^{1,p_{i}}(\Omega),$ for every $q\geq1$ there
exists a constant $C$, depending on $\left\vert \Omega\right\vert $ and $q$,
such that
\begin{equation}
\left\Vert u\right\Vert _{L^{q}(\Omega)}\leq\ C\;\left\Vert \partial_{x_{i}%
}u\right\Vert _{L^{q}(\Omega)}. \label{dis poincare}%
\end{equation}
%\vspace {-5.00 mm}
We set $W_{0}^{1,\overrightarrow{p}}(\Omega)=\displaystyle\bigcap_{i=1}%
^{N}W_{0}^{1,p_{i}}(\Omega)$ with the norm
\begin{equation}
\left\Vert u\right\Vert _{W_{0}^{1,\overrightarrow{p}}(\Omega)}=\overset
{N}{\underset{i=1}{\sum}}\left\Vert \partial_{x_{i}}u\right\Vert _{L^{p_{i}%
}(\Omega)} \label{Sob_norm}%
\end{equation}
%\vspace {-5.00 mm}
and we denote its dual by $\left(  W_{0}^{1,\overrightarrow{p}}(\Omega
)\right)  ^{\prime}.$

Moreover we put \ $L^{\overrightarrow{p}}\left(  0,T;W_{0}^{1,\overrightarrow
{p}}(\Omega)\right)  =\displaystyle\bigcap_{i=1}^{N}L^{p_{i}}\left(
0,T;W_{0}^{1,p_{i}}(\Omega)\right)  $ equipped with the following norm
%equipped with norm%
\begin{equation}
\Vert u\Vert_{L^{\overrightarrow{p}}\left(  0,T;W_{0}^{1,\overrightarrow{p}%
}(\Omega)\right)  }=\sum_{i=1}^{N}\left(  \int_{0}^{T}\left\Vert u_{x_{i}%
}\right\Vert _{L^{p_{i}}(\Omega)}^{p_{i}}dt\right)  ^{\frac{1}{p_{i}}}.
\label{norm}%
\end{equation}
On denoting by $\bar{p}$ the \textit{harmonic mean} of $p_{1},\ldots,p_{N}$,
\emph{i.e.}%

\begin{equation}
\frac{1}{\bar{p}}=\frac{1}{N}\overset{N}{\underset{i=1}{\sum}}\frac{1}{p_{i}},
\label{p barrato bis}%
\end{equation}
a Sobolev-type inequality tells us that whenever $u$ belongs to $W_{0}%
^{1,\overrightarrow{p}}(\Omega)$, there exists a constant $C_{S}$ such that
%\begin{equation}
%\left\Vert u\right\Vert _{L^{q}(\Omega)}\leq C_{S}\;\overset{N}{\underset
%{i=1}{{\displaystyle\prod}}}\left\Vert \partial_{x_{i}}u\right\Vert
%_{L^{p_{i}}(\Omega)}^{\frac{1}{N}}, \label{sobolev 2}%
%\end{equation}
%where $q=\bar{p}^{\ast}=\frac{N\bar{p}}{N-\bar{p}}$ if
%$\bar{p}<N$ or $q\in\left[  1,+\infty\right[  $ if $\bar{p}\geq N$.
%On the right-hand side of (\ref{sobolev 2}) it is possible to replace the
%geometric mean by the arithmetic mean: let $a_{1},...,a_{N}$ be positive
%numbers$,$ it holds%
%\begin{equation}
%\overset{N}{\underset{i=1}{{\displaystyle\prod}}}a_{i}^{\frac{1}{N}}\leq
%\frac{1}{N}\overset{N}{\underset{i=1}{\sum}}a_{i}, \label{geo-arti}%
%\end{equation}
%which implies by (\ref{sobolev 2})%
\begin{equation}
\left\Vert u\right\Vert _{L^{q}(\Omega)}\leq C_{S}\;\overset{N}{\underset
{i=1}{\sum}}\left\Vert \partial_{x_{i}}u\right\Vert _{L^{p_{i}}(\Omega)}
\label{sobolev}%
\end{equation}
where $q=\overline{p}^{\ast}=\frac{N\overline{p}}{N-\overline{p}}$ if
$\overline{p}<N$ or $q\in\left[  1,+\infty\right[  $ if $\overline{p}\geq N$
(see \cite{troisi}). If in plus $\bar{p}<N$, inequality (\ref{sobolev})
implies the continuous embedding of the space $W_{0}^{1,\overrightarrow{p}%
}(\Omega)$ into $L^{q}(\Omega)$ for every $q\in\lbrack1,\bar{p}^{\ast}]$. On
the other hand, the continuity of the embedding $W_{0}^{1,\overrightarrow{p}%
}(\Omega)\subset L^{p_{+}}(\Omega)$ with $p_{+}:=\max\{p_{1},\ldots,p_{N}\}$
relies on inequality (\ref{dis poincare}). It may happen that $\bar{p}^{\ast
}<p_{+}$ if the exponents $p_{i}$ are not closed enough. Then $p_{\infty
}:=\max\{\bar{p}^{\ast},p_{+}\}$ turns out to be the critical exponent in the
anisotropic Sobolev embedding.

\noindent

\bigskip

\subsection{Symmetrization}

A precise statement of our results requires the use of classical notions of
rearrangement and of suitable symmetrization of a Young function, introduced
by Klimov in \cite{Klimov 74}. \newline Let $u$ be a measurable function
(continued by $0$ outside its domain) fulfilling
\begin{equation}
\left\vert \{x\in\mathbb{R}^{N}:\left\vert u(x)\right\vert >t\}\right\vert
<+\infty\text{ \ \ \ for every }t>0. \label{insieme livello di misura finita}%
\end{equation}
The \textit{symmetric decreasing rearrangement} of $u$ is the function
$u^{\bigstar}:\mathbb{R}^{N}\rightarrow\left[  0,+\infty\right[  $
$\ $satisfying
\begin{equation}
\{x\in\mathbb{R}^{N}:u^{\bigstar}(x)>t\}=\{x\in\mathbb{R}^{N}:\left\vert
u(x)\right\vert >t\}^{\bigstar}\text{ \ for }t>0. \label{livello palla}%
\end{equation}
The \textit{decreasing rearrangement} $u^{\ast}$ of $u$ is defined as
\[
u^{\ast}(s)=\sup\{t>0:\mu_{u}(t)>s\}\text{ \ for }s\geq0,
\]
where
\[
\mu_{u}(t)=\left\vert \{x\in{\Omega}:\left\vert u(x)\right\vert
>t\}\right\vert \text{ \ \ \ \ for }t\geq0
\]
denotes the \textit{distribution function} of $u$. \newline Moreover
\[
u^{\bigstar}(x)=u^{\ast}(\omega_{N}\left\vert x\right\vert ^{N})\text{
\ \ }\hbox{\rm for a.e.}\;x\in{{\mathbb{R}}^{N}.}%
\]
Analogously, we define the \textit{symmetric increasing rearrangement}
$u_{\bigstar}$ on replacing \textquotedblleft$>$\textquotedblright\ by
\textquotedblleft$<$\textquotedblright\ in the definitions of the sets in
(\ref{insieme livello di misura finita}) and (\ref{livello palla}). Moreover,
we set
\[
u^{\ast\ast}(s)=\frac{1}{s}\int_{0}^{s}u^{\ast}(r)\;dr\qquad
\hbox{\rm for}\;s>0.
\]
We refer to \cite{BS} for details on these topics.

We just recall the following property of rearrangements which will be useful
in the following (see for example \cite{ALT1}):

\begin{lemma}
\label{riordsum} If $f,g$ are measurable functions defined in $\Omega$, then
\[
\int_{0}^{r}(f+g)^{*}(s)\,ds\leq\int_{0}^{r} f^{*}(s)+g^{*}(s)\,ds,\quad
\forall\, r\in[0,|\Omega|].
\]

\end{lemma}

\bigskip

\bigskip

In this paper we will consider an $N-$\textit{dimensional Young function}
(namely an even convex function such that $\Phi\left(  0\right)  =0$ \ \ and
$\underset{\left\vert \xi\right\vert \rightarrow+\infty}{\lim}\Phi\left(
\xi\right)  $ $=+\infty)$ of the following type:%
\begin{equation}
\Phi\left(  \xi\right)  =\underset{i=1}{\overset{N}{%
%TCIMACRO{\dsum }%
%BeginExpansion
{\displaystyle\sum}
%EndExpansion
}}\alpha_{i}\left\vert \xi_{i}\right\vert ^{p_{i}}\text{ \ for }\xi
\in\mathbb{R}^{N}\text{ with }\alpha_{i}>0\text{ for }i=1,...,N. \label{Fi pi}%
\end{equation}

We denote by $\Phi_{\blacklozenge}:\mathbb{R\rightarrow}\left[  0,+\infty
\right[  $ the symmetrization of $\Phi$ introduced in \cite{Klimov 74}. It is
the one-dimensional Young function fulfilling%
\begin{equation}
\Phi_{\blacklozenge}(\left\vert \xi\right\vert )=\Phi_{\bullet\bigstar\bullet
}\left(  \xi\right)  \text{ \ for }\xi\in\mathbb{R}^{N}, \label{def fi rombo}%
\end{equation}
where $\Phi_{\bullet}$ is the Young conjugate function of $\Phi$ given by%
\[
\Phi_{\bullet}\left(  \xi^{\prime}\right)  =\sup\left\{  \xi\cdot\xi^{\prime
}-\Phi\left(  \xi\right)  :\xi\in\mathbb{R}^{N}\right\}  \text{ \ \ for \ }%
\xi^{\prime}\in\mathbb{R}^{N}.
\]
So $\Phi_{\blacklozenge}$ is the composition of Young conjugation, symmetric
increasing rearrangement and Young conjugate again. Easy calculations show
(see \textit{e.g. }\cite{cianchi anisotropo}), that
\begin{equation}
\Phi_{\blacklozenge}(\left\vert \xi\right\vert )=\Lambda\left\vert
\xi\right\vert ^{\overline{p}}, \label{fi rombo}%
\end{equation}
where $\overline{p}$ is the harmonic mean of exponents $p_{1},\ldots,p_{N}$
defined in (\ref{p barrato bis}) and
\begin{equation}
\Lambda=\frac{2^{\overline{p}}\left(  \overline{p}-1\right)  ^{\overline{p}%
-1}}{\overline{p}^{\overline{p}}}\left[  \frac{\underset{i=1}{\overset{N}{\Pi
}}p_{i}^{\frac{1}{p_{i}}}\left(  p_{i}^{\prime}\right)  ^{\frac{1}%
{p_{i}^{\prime}}}\Gamma(1+1/p_{i}^{\prime})}{\omega_{N}\Gamma(1+N/\overline
{p}^{\prime})}\right]  ^{\frac{\overline{p}}{N}}\left(  \underset
{i=1}{\overset{N}{\Pi}}\alpha_{i}^{\frac{1}{p_{i}}}\right)  ^{\frac
{\overline{p}}{N}} \label{lapda}%
\end{equation}
with $\omega_{N}$ the measure of the $N-$dimensional unit ball, $\Gamma$ the
Gamma function and $p_{i}^{\prime}=\frac{p_{i}}{p_{i}-1}$ with the usual
conventions if $p_{i}=1$. \newline We remember that in the anisotropic setting
a \textit{Polya-Szeg\"{o} principle} holds (see \cite{cianchi anisotropo}).
Let $u$ be a weakly differentiable function in $%
%TCIMACRO{\U{211d} }%
%BeginExpansion
\mathbb{R}
%EndExpansion
^{N}$ satisfying (\ref{insieme livello di misura finita}) and such that
$\underset{i=1}{\overset{N}{%
%TCIMACRO{\dsum }%
%BeginExpansion
{\displaystyle\sum}
%EndExpansion
}}\alpha_{i}%
%TCIMACRO{\dint _{\mathbb{R}^{N}}}%
%BeginExpansion
{\displaystyle\int_{\mathbb{R}^{N}}}
%EndExpansion
\left\vert \frac{\partial u}{\partial x_{i}}\right\vert ^{p_{i}}dx<+\infty,$
then $u^{\bigstar}$ is weakly differentiable in $%
%TCIMACRO{\U{211d} }%
%BeginExpansion
\mathbb{R}
%EndExpansion
^{N}$ and
\begin{equation}
\Lambda\int_{\mathbb{R}^{N}}\left\vert \nabla u^{\bigstar}\right\vert
^{\overline{p}}dx\leq\underset{i=1}{\overset{N}{%
%TCIMACRO{\dsum }%
%BeginExpansion
{\displaystyle\sum}
%EndExpansion
}}\alpha_{i}\int_{\mathbb{R}^{N}}\left\vert \frac{\partial u}{\partial x_{i}%
}\right\vert ^{p_{i}}dx\text{ .\ } \label{polya}%
\end{equation}

\bigskip

\section{\textbf{Main results}}

%%%%%%%%%%%%%%%%%%%%%%%%%%%%%%%%%%%%%%%%%%%%%%%%%%%%%%%%%%%%%%%%%%%%%%%%%%%%%%%%%%%%%%%%

\label{sec3}

%DA\ FARE:\ Se si vuole fare il parabolico con $\overline{p}\neq2,$ bisogna
%fare il confronto di%
%\[
%\left\{
%\begin{array}
%[c]{lll}%
%-\operatorname{div}(a(x,\nabla u))+cu=f &  & \text{in }\Omega\\
%u=0 &  & \text{on }\partial\Omega,
%\end{array}
%\right.
%\]
%con $c>0.$
We deal with a class of nonlinear parabolic problems subject to general growth
conditions and having the form
\begin{equation}
\left\{
\begin{array}
[c]{lll}%
u_{t}-\operatorname{div}(a(x,t,u,\nabla u))=f(x,t) & \quad\text{in }%
Q_{T}:=\Omega\times\left(  0,T\right)  & \\
u(x,0)=u_{0}(x) & \quad\text{in}\;\Omega & \\
u(x,t)=0 & \quad\text{on}\;\partial\Omega\times\left(  0,T\right)  , &
\end{array}
\right.  \label{prob_parab}%
\end{equation}
where $\Omega$ is an open, bounded subset of $\mathbb{R}^{N}$ with Lipschitz
continuous boundary, $N\geq2,$ $a:Q_{T}\times\mathbb{R}\times\mathbb{R}%
^{N}\rightarrow\mathbb{R}^{N}$ is a Carath\'{e}odory function such that, for
\textit{a.e.} $(x,t)\in Q_{T}$, for all $s\in\mathbb{R}$ and for all $\xi
,\xi^{\prime}\in{{\mathbb{R}}^{N}}$,

\begin{itemize}
\item[(H1)] { $a(x,t,s,\xi)\cdot\xi\geq\overset{N}{\underset{i=1}{\sum}}%
\alpha_{i}\left\vert \xi_{i}\right\vert ^{p_{i}}\quad\hbox{with
$\alpha_{i}>0$};$}

\item[(H2)] { $\left\vert a_{j}(x,t,s,\xi)\right\vert \leq\beta{{}}\left[
|s|^{\bar{p}/p_{j}^{\prime}}+\left\vert \xi_{j}\right\vert ^{p_{j}-1}\right]
\quad\hbox{ with $\beta>0 \quad \forall j=1, \ldots,
N$};$}

\item[(H3)] { $\left\vert a_{j}(x,t,s,\xi)-a_{j}(x,t,s^{\prime},\xi
)\right\vert \leq\gamma\left\vert \xi_{j}\right\vert ^{p_{j}-1}|s-s^{\prime}|$
\quad\hbox{with $\gamma>0 \quad \forall j=1,
\ldots, N$};}

\item[(H4)] { $\left(  a(x,t,s,\xi)-a(x,t,s,\xi^{\prime})\right)  \cdot\left(
\xi-\xi^{\prime}\right)  >0$ }$\quad${with }$\xi\neq\xi^{\prime}.${$\qquad$
%\begin{equation}
%\left\{
%\begin{array}
%[c]{lll}%
%u_{t}-\left(  a_{ij}(x,t)\left\vert u_{x_{i}}\right\vert ^{p_{i}-2}u_{x_{i}%
%}\right)  _{x_{j}}=g(x,t) & \text{in} & Q:=\Omega\times\lbrack0,T]\\
%u(x,0)=u_{0}(x) & \text{in} & \Omega\\
%u(x,t)=0 & \text{on} & \partial\Omega\times\lbrack0,T]
%\end{array}
%\right.  \label{prob_parab}%
%\end{equation}
}
\end{itemize}

Moreover, we assume that

\begin{itemize}
\item[(H5)] { $f\in\overset{N}{\underset{i=1}{\sum}}L^{{p_{i}^{\prime}}%
}(0,T,W^{-1,p_{i}^{\prime}}(\Omega)+L^{2}(\Omega))\qquad\hbox{and}\qquad
u_{0}\in L^{2}(\Omega).$}
\end{itemize}

Here, $1\leq p_{1},\ldots,p_{N}<\infty$ and ${{\bar{p}}}$ denotes the harmonic
mean of $p_{1},\ldots,p_{N},$ defined in (\ref{p barrato bis}), such that
$\overline{p}>1$.

\begin{definition}
We say that a function $u\in L^{\overrightarrow{p}}\left(  0,T;W_{0}%
^{1,\overrightarrow{p}}(\Omega)\right)  \cap\,C\left(  0,T;L^{2}%
(\Omega)\right)  $ is a weak solution to problem (\ref{prob_parab}) if for all
$t\in\left(  0,T\right)  $
\begin{align}
\int_{\Omega}u(x,t)\varphi(x,t)\;dx  &  +\int_{0}^{t}\int_{\Omega}%
(-u(x,\tau)\varphi_{\tau}(x,\tau)+a(x,\tau,u,\nabla u)\cdot\nabla
\varphi(x,\tau))\;dx\;d\tau\label{def_sol}\\
&  =\int_{\Omega}u_{0}(x)\varphi(x,0)\;dx+\int_{0}^{t}\int_{\Omega}%
f(x,\tau)\varphi(x,\tau)\;dx\;d\tau\nonumber
\end{align}
for any $\varphi\in W^{1,2}(0,T;L^{2}(\Omega))\cap L^{\overrightarrow{p}%
}(0,T;W_{0}^{1,\overrightarrow{p}}(\Omega)).$
\end{definition}

Since $A(u)=-\operatorname{div}(a(x,t,u,\nabla u))$ is a pseudomonotone and
coercive operator acting between $L^{\overrightarrow{p}}(0,T;W_{0}%
^{1,\overrightarrow{p}}(\Omega)\cap L^{2}(\Omega))$ and $\overset{N}%
{\underset{i=1}{\sum}}L^{{p_{i}^{\prime}}}(0,T,W^{-1,p_{i}^{\prime}}%
(\Omega)+L^{2}(\Omega)),$ it is well-known (see \cite{Ls} and
\cite{antontsev-chipot-08}) that there exists a unique weak solution to
problem (\ref{prob_parab}).

\noindent Our aim is to obtain a comparison between concentrations of the
solution $u$ to problem (\ref{prob_parab}) and the solution $v$ to problem
(\ref{symm_Pr}), which has a unique weak solution $v\in L^{\bar{p}}\left(
0,T;W_{0}^{1,\bar{p}}(\Omega)\right)  \cap C\left(  0,T;L^{2}(\Omega)\right)
$.

In this section we adopt the following convention: if $h(x,t)$ is defined in
$Q_{T},$ we denote by $h^{\ast}(\sigma,t)$ the decreasing rearrangement of $h$
with respect to $x$ for $t$ fixed.

\bigskip

\begin{theorem}
\label{Th parabolico}Assume that (H1)-(H5) hold. Let $u$ be the weak solution
to problem \eqref{prob_parab} and $v$ be the solution to problem
\eqref{symm_Pr}, then we have%
\begin{equation}
\int_{0}^{s}u^{\ast}(\sigma,t)\;d\sigma\leq\int_{0}^{s}v^{\ast}(\sigma
,t)\;d\sigma\qquad\hbox{$x \in (0, |\Omega|)$  for a.e. $t \in
(0,T)$.} \label{e:confr}%
\end{equation}

\end{theorem}

\bigskip

The following result is a slight extension of Theorem \ref{Th parabolico} when
the datum in problem (\ref{symm_Pr}) is not the rearrangement of datum $f$ of
problem (\ref{prob_sym_zero}), but it is a function that dominates $f$.

\bigskip

\begin{corollary}
\label{corollario parabolico}Assume the same hypothesis of Theorem
\ref{Th parabolico}. Let $u$ be the weak solution to problem
\eqref{prob_parab} and $v$ be the solution to the following problem%
\begin{equation}
\left\{
\begin{array}
[c]{ll}%
v_{t}-\mathrm{{div}}\left(  \Lambda|\nabla v|^{\bar{p}-2}\nabla v\right)
=\widetilde{f}(x,t) & \quad
\hbox{in $Q^{\bigstar}:= \Omega^{\bigstar}\times (0, T)$}\\
v(x,0)=\widetilde{u}_{0}(x) & \quad\hbox{in $\Omega^{\bigstar}$}\\
v(x,t)=0 & \quad\hbox{on $\partial \Omega\times (0, T).$}
\end{array}
\right.  \label{prob_sim_corol}%
\end{equation}
where $\widetilde{f}=\widetilde{f}^{\bigstar}$ and $\widetilde{u}%
_{0}=\widetilde{u}_{0}^{\bigstar}$ are functions such that for \textit{a.e.}
$t\in(0,T)$%
\[
\int_{0}^{s}f^{\ast}(\sigma,t)d\sigma\leq\int_{0}^{s}\widetilde{f}^{\ast
}(\sigma,t)d\sigma\text{ \ \ \ for }s\in\left[  0,\left\vert \Omega\right\vert
\right]
\]
and%
\[
\int_{0}^{s}u_{0}^{\ast}(\sigma)d\sigma\leq\int_{0}^{s}\widetilde{u}_{0}%
^{\ast}(\sigma)d\sigma\text{ \ \ \ for }s\in\left[  0,\left\vert
\Omega\right\vert \right]  ,
\]
respectively. Then we have
\[
\int_{0}^{s}u^{\ast}(\sigma,t)\;d\sigma\leq\int_{0}^{s}v^{\ast}(\sigma
,t)\;d\sigma\qquad\hbox{$x \in (0, |\Omega|)$  for a.e. $t \in
(0,T)$.}
\]

\end{corollary}

\bigskip

Using Corollary \ref{corollario parabolico} it is possible to prove the
following estimates of the solution $u\left(  \cdot,t\right)  $ to problem
(\ref{prob_parab}) in term of the solution $v\left(  \cdot,t\right)  $ to
problem (\ref{prob_sim_corol}).

\begin{corollary}
\label{Corollario}Assume the same hypothesis \ref{corollario parabolico}. If
$u$ is the weak solution to problem (\ref{prob_parab}) and $v$ is the solution
to problem (\ref{prob_sim_corol}), then we have
\[
\left\Vert u\left(  \cdot,t\right)  \right\Vert _{L^{p,q}(\Omega)}%
\leq\left\Vert v\left(  \cdot,t\right)  \right\Vert _{L^{p,q}(\Omega
^{\bigstar})}\text{ \ for }t>0,
\]
where $1\leq p<\infty,1\leq q\leq\infty$ and
\[
\left\Vert h\right\Vert _{L^{p,q}(\Omega)}=\left\{
\begin{array}
[c]{ccc}%
\left[  {\displaystyle\int_{0}^{\left\vert \Omega\right\vert }}\left(
s^{\frac{1}{p}}\,h^{\ast\ast}(s)\right)  ^{q}\;\displaystyle\frac{ds}%
{s}\right]  ^{\frac{1}{q}} & \qquad
\hbox{if $1\leq p<+\infty,\, 1\leq q<\infty$} & \\
\displaystyle\sup_{s\in\left(  0,\left\vert \Omega\right\vert \right)
}s^{\frac{1}{p}}\,h^{\ast\ast}(s) & \qquad\hbox{if  $1\leq p <+\infty,\,
q=\infty$}. &
\end{array}
\right.
\]

\end{corollary}

\bigskip

Let us consider\ a weak solution $u\in L_{loc}^{\overrightarrow{p}}%
(0,+\infty,W_{0}^{1,\overrightarrow{p}})\cap C(0,+\infty,L^{2}(\Omega))$ to
the following problem%
\begin{equation}
\left\{
\begin{array}
[c]{ll}%
u_{t}-\overset{N}{\underset{i=1}{\sum}}\left(  \alpha_{i}\left\vert
\partial_{x_{i}}u\right\vert ^{p_{i}-2}\partial_{x_{i}}u\right)  _{x_{i}%
}=0\quad & \mbox{in}\text{ }\left(  0,+\infty\right)  \times\Omega\\
u(t,x)=0 & \mbox{on}\text{ }\left(  0,+\infty\right)  \times\partial\Omega,\\
u(0,x)=u_{0}(x)\quad & \mbox{in}\text{ }\Omega,
\end{array}
\right.  \label{prob decadi}%
\end{equation}
with $\bar{p}=2.$

As a consequence of Theorem \ref{Th parabolico} we study the asymptotic
behavior of solution $u$ to problem (\ref{prob decadi}) as time variable $t$
goes to infinity. Proceeding as in \cite{Ta3}, it is possible to show that all
the solutions to problem decay exponentially to zero as time goes to infinity.

\bigskip

\begin{corollary}
\label{Corollario2}Assume the same hypothesis of Theorem \ref{Th parabolico}.
If $\lambda$\ is the smallest eigenvalue of the following Sturm-Lionville
problem%
\begin{equation}
\left\{
\begin{array}
[c]{ll}%
-\chi^{\prime\prime}\left(  r\right)  +\frac{n-1}{r}\chi^{\prime}\left(
r\right)  =\lambda\chi\left(  r\right)  \quad & \mbox{in}\text{ }\left(
0,R_{\Omega}\right) \\
\chi^{\prime}(0)=\chi(R_{\Omega})=0 &
\end{array}
\right.  \label{S-lionville}%
\end{equation}
and $u$ is a non-zero solution to problem (\ref{prob decadi}), then we have
\[
\left\Vert u(t,\cdot)\right\Vert _{L^{2}(\Omega)}\leq e^{-\lambda t}\left\Vert
u(0,\cdot)\right\Vert _{L^{2}(\Omega)}\text{ \quad for }t>0.
\]

\end{corollary}

\bigskip

In order to prove Theorem \ref{Th parabolico} we use the well-known
discretization's method. To this purpose, we divide $[0,T]$ into $M$
subintervals
\[
0=t_{0}<t_{1}<...<t_{M}=T
\]
with $t_{i+1}-t_{i}\leq\delta(M),$ where $\delta(M)\rightarrow0$ as
$M\rightarrow+\infty.$ So one can approximate the solution $u$ to problem
\eqref{prob_parab} by the sequence $\{u_{M}\}_{M}$ of functions defined in
terms of the initial datum $u_{0}$ and the weak solution to the elliptic
problem
\begin{equation}
\left\{
\begin{array}
[c]{ll}%
-\hbox{div}\left(  a^{m}\left(  x,U,\nabla U\right)  \right)
+\displaystyle\frac{U}{t_{m+1}-t_{m}}=f^{m}(x)+\frac{u^{m-1}}{t_{m+1}-t_{m}} &
\hbox{\rm in $\Omega$}\\
& \\
U=0 & \hbox{\rm on $\partial \Omega$},
\end{array}
\right.  \label{discr_pr}%
\end{equation}
where
\begin{align*}
a^{m}(x,s,\xi)  &  =\frac{1}{t_{m+1}-t_{m}}\int_{t_{m}}^{t_{m+1}}%
a(x,t,s,\xi)\;dt\\
\text{ }f^{m}(x)  &  =\frac{1}{t_{m+1}-t_{m}}\int_{t_{m}}^{t_{m+1}}g(x,t)\;dt.
\end{align*}
More precisely,
\begin{equation}
u_{M}(x,t)=\left\{
\begin{array}
[c]{ll}%
u^{0}(x) & \quad\hbox{if $t\in [0, t_1[$}\\
& \\
u^{m}(x) & \quad\hbox{if $t\in [t_m, t_{m+1}[$ and $1\leq m\leq
M-1$,}
\end{array}
\right.  \label{UN}%
\end{equation}
where $u^{0}(x)$ coincides with $u_{0}(x)$ for $x\in\Omega$, and $u^{m}(x)$
for $1\leq m\leq M-1$ denotes the weak solution to problem \eqref{discr_pr}.

Analogously, the solution $v$ to problem \eqref{symm_Pr} can be approximated
by the sequence $\{v_{M}\}_{M}$ of functions
\begin{equation}
v_{M}(x,t)=\left\{
\begin{array}
[c]{ll}%
v^{0}(x) & \quad\hbox{if $t\in [0, t_1[$}\\
& \\
v^{m}(x) & \quad\hbox{if $t\in [t_m, t_{m+1}[$ and $1\leq m\leq
M-1$,}
\end{array}
\right.  \label{VN}%
\end{equation}
where $v^{0}(x)$ agrees with $u_{0}^{\bigstar}(x)$ for $x\in\Omega$, and
$v^{m}(x)$ for $1\leq m\leq M-1$ is the weak solution to the elliptic problem
\begin{equation}
\left\{
\begin{array}
[c]{ll}%
-\hbox{div}\left(  \Lambda|\nabla V|^{\bar{p}-2}\nabla V\right)
+\displaystyle\frac{V}{t_{m+1}-t_{m}}=(f^{m})^{\bigstar}(x)+\frac{v^{m-1}%
}{t_{m+1}-t_{m}} & \hbox{in}\;\;\Omega^{\bigstar}\\
& \\
V=0 & \hbox{on}\;\;\partial\Omega^{\bigstar}.
\end{array}
\right.  \label{discr_symm_Pr}%
\end{equation}

At this point to prove Theorem \ref{Th parabolico}, we begin by checking a
comparison result for elliptic problem (\ref{discr_pr}) that we will present
in the next subsection.

\subsection{Comparison result for elliptic problem}

In the present subsection we focus our attention to the following class of
anisotropic elliptic problems
\begin{equation}
\left\{
\begin{array}
[c]{ll}%
-\operatorname{div}(a(x,w,\nabla w))+\lambda w(x)=g(x) & \hbox{in $\Omega$}\\
& \\
w=0 & \hbox{on $\partial\Omega$},
\end{array}
\right.  \label{problema_zero}%
\end{equation}
where $\Omega$ is a bounded open subset of ${{\mathbb{R}}^{N}}$ with Lipschitz
continuous boundary, $N\geq2,$ $a:\Omega\times\mathbb{R}\times\mathbb{R}%
^{N}\rightarrow\mathbb{R}^{N}$ is a Carath\'{e}odory function such that for
\textit{a.e.} $x\in\Omega,\hbox{ for all $s\in \rN$ and for
all $ \xi, \xi' \in \rN$}$

\begin{itemize}
\item[(A1)] { $a(x,s,\xi)\cdot\xi\geq\overset{N}{\underset{i=1}{\sum}}%
\alpha_{i}\left\vert \xi_{i}\right\vert ^{p_{i}}\ \ \text{with }\alpha_{i}%
>0;$}

\item[(A2)] { $\left\vert a_{j}(x,s,\xi)\right\vert \leq\beta{{}}\left[
|s|^{\bar{p}/p_{j}^{\prime}}+\left\vert \xi_{j}\right\vert ^{p_{j}-1}\right]
\quad\hbox{ with $\beta>0 \quad \forall j=1, \ldots,
N$};$}

\item[(A3)] { $\left\vert a_{j}(x,t,s,\xi)-a_{j}(x,t,s^{\prime},\xi
)\right\vert \leq\gamma\left\vert \xi_{j}\right\vert ^{p_{j}-1}|s-s^{\prime}|$
\quad\hbox{with $\gamma>0 \quad \forall j=1,
\ldots, N$}}

\item[(A4)] {$\left(  a(x,s,\xi)-a(x,s,\xi^{\prime})\right)  \cdot\left(
\xi-\xi^{\prime}\right)  >0\qquad\hbox{for  $ \xi\neq \xi'$.}$}
\end{itemize}

Moreover

\begin{itemize}
\item[(A5)] {$\lambda>0$ and $g\in$}$\left(  W_{0}^{1,\overrightarrow{p}%
}(\Omega)\right)  ^{\prime}${.}

Here $1\leq p_{1},\ldots,p_{N}<\infty$ and $\bar{p}$ is the harmonic mean of
$p_{1},\ldots,p_{N},$ defined in (\ref{p barrato bis}), such that
$\overline{p}>1.$
\end{itemize}

\bigskip

We are interested in proving a comparison result between the concentration of
the solution $w\in W_{0}^{1,\overrightarrow{p}}(\Omega)$ to problem
\eqref{problema_zero} and the solution $z\in W_{0}^{1,\bar{p}}(\Omega
^{\bigstar})$ to the following problem
\begin{equation}
\left\{
\begin{array}
[c]{ll}%
-\operatorname{div}(\Lambda|\nabla z|^{\bar{p}-2}\nabla z)+\lambda
z(x)=g^{\bigstar}(x) & \hbox{in $\Omega^\bigstar$}\\
& \\
z=0 & \hbox{on $\partial\Omega^\bigstar$}.
\end{array}
\right.  \label{prob_sym_zero}%
\end{equation}

For this kind of results see also \cite{AdBF2} and \cite{AdBF3}.

\bigskip

We emphasize that under our assumptions there exists a unique bounded weak
solution (by a slight modification of classical results see
\emph{\textit{e.g.} }\cite{castro2}, \cite{B}\emph{ }and \ see
\cite{antontsev-chipot-08} as regard the uniqueness).
%%%%%%%%%%%%%%%%%%%%%%%%%%%%%%%%%%%%%%%%%%%%%%%%%%%%%%%%%%%%%%%%%%%%%%%%%%%%%%%

\bigskip

\begin{theorem}
\label{ellipticcomp} Assume that (A1)- (A5) hold. If $w$ is the weak solution
to problem \eqref{problema_zero} and $z$ is the weak solution to problem
\eqref{prob_sym_zero}, then we have
\[
\int_{0}^{s}w^{\ast}(\sigma)\,d\sigma\leq\int_{0}^{s}z^{\ast}(\sigma
)\,d\sigma,\quad\forall s\in\lbrack0,|\Omega|].
\]

\end{theorem}

\noindent

\begin{proof}
We choose the functions $w_{\kappa,\tau}:\Omega\rightarrow$\textbf{
}$\mathbb{R}$ defined as
\[
w_{\kappa,\tau}\left(  x\right)  =\left\{
\begin{array}
[c]{ll}%
0 & \quad\mbox{ if }\left\vert w\left(  x\right)  \right\vert \leq\tau
,\vspace{0.2cm}\\
\left(  \left\vert w\left(  x\right)  \right\vert -\tau\right)  \text{sign}%
\left(  w\left(  x\right)  \right)  & \quad\mbox{ if }\tau<\left\vert w\left(
x\right)  \right\vert \leq\tau+\kappa\\
& \\
\kappa\;\text{sign}\left(  w\left(  x\right)  \right)  & \quad\mbox{ if }\tau
+\kappa<\left\vert w\left(  x\right)  \right\vert
\end{array}
\right.
\]
for any fixed $\ \tau$ and $\kappa>0$, as test function in problem
(\ref{problema_zero}) and by (A1), we get%

\begin{align}
\frac{1}{\kappa}\overset{N}{\underset{i=1}{%
%TCIMACRO{\dsum }%
%BeginExpansion
{\displaystyle\sum}
%EndExpansion
}}\alpha_{i}\int_{\tau<\left\vert w\right\vert <\tau+\kappa}\left\vert
\frac{\partial w}{\partial x_{i}}\right\vert ^{p_{i}}dx  &  \leq\frac
{1}{\kappa}\int_{\tau<\left\vert w\right\vert <\tau+\kappa}a(x,w,\nabla
w)dx=\label{def_1}\\
&  =\frac{1}{\kappa}\int_{\tau<\left\vert w\right\vert <\tau+\kappa}\left(
\lambda w(x)+g(x)\right)  \left(  \left\vert w\left(  x\right)  \right\vert
-\tau\right)  sign\left(  w\left(  x\right)  \right)  dx\nonumber\\
&  +\int_{\left\vert w\right\vert >\tau+\kappa}\left(  \lambda
w(x)+g(x)\right)  sign\left(  w\left(  x\right)  \right)  dx.\nonumber
\end{align}

Arguing as in \cite{cianchi anisotropo}, we can apply Polya-Szeg\"{o}
principle (\ref{polya}) to function $w_{\kappa,\tau}$ continued by $0$ outside
$\Omega$ taking into account \eqref{Fi
pi} and \eqref{fi rombo}. \ We obtain%
\begin{equation}
\overset{N}{\underset{i=1}{%
%TCIMACRO{\dsum }%
%BeginExpansion
{\displaystyle\sum}
%EndExpansion
}}\alpha_{i}\int_{\tau<\left\vert w\right\vert <\tau+\kappa}\left\vert
\frac{\partial w}{\partial x_{i}}\right\vert ^{p_{i}}dx=\overset{N}%
{\underset{i=1}{%
%TCIMACRO{\dsum }%
%BeginExpansion
{\displaystyle\sum}
%EndExpansion
}}\alpha_{i}\int_{\mathbb{R}^{N}}\left\vert \frac{\partial w_{\kappa,\tau}%
}{\partial x_{i}}\right\vert ^{p_{i}}dx\geq\Lambda\int_{\mathbb{R}^{N}%
}\left\vert \nabla w_{\kappa,\tau}^{\bigstar}\right\vert ^{\overline{p}%
}dx=\Lambda\int_{\tau<w_{\kappa,\tau}^{\bigstar}<\tau+\kappa}\left\vert \nabla
w_{\kappa,\tau}^{\bigstar}\right\vert ^{\overline{p}}dx. \label{C}%
\end{equation}
By (\ref{def_1}) and (\ref{C}), letting $\kappa\rightarrow0$ we get
\[
-\frac{d}{d\tau}\int_{w^{\bigstar}>\tau}\Lambda|\nabla w^{\bigstar}|^{\bar{p}%
}\;dx\leq\int_{|w|>\tau}\left(  |g(x)|+\lambda|w(x)|\right)  \;dx\qquad
\hbox{for $ a.e.$ $\tau >0$}.
\]
Using Coarea formula and H\"{o}lder's inequality, we can write
\[
\left(  -\frac{d}{d\tau}\int_{w^{\bigstar}>\tau}|\nabla w^{\bigstar}|^{\bar
{p}}\;dx\right)  ^{\frac{1}{\bar{p}}}\geq N\omega_{N}^{\frac{1}{N}}\mu
_{w}(\tau)^{\frac{1}{N^{\prime}}}\left(  -\mu_{w}^{\prime}(\tau)\right)
^{-\frac{1}{\bar{p}^{\prime}}}\qquad\hbox{for $ a.e.$ $\tau >0$},
\]
where $\mu_{w}(\tau)=\left\vert \{x\in{\Omega}:\left\vert w(x)\right\vert
>\tau\}\right\vert .$ By Hardy-Littlewood inequality we obtain
\begin{equation}
\Lambda\left(  N\omega_{N}^{\frac{1}{N}}\mu_{w}(\tau)^{\frac{1}{N^{\prime}}%
}\left(  -\mu_{w}^{\prime}(\tau)\right)  ^{-\frac{1}{\bar{p}^{\prime}}%
}\right)  ^{\bar{p}}\leq\int_{0}^{\mu_{w}(\tau)}\left(  \lambda w^{\ast
}(s)+g^{\ast}(s)\right)  \;ds\qquad\hbox{for $a.e.$ $\tau >0$}. \label{E:4}%
\end{equation}
Putting
\[
\label{E:5}\mathcal{W}(s)=\int_{0}^{s}\lambda w^{\ast}(\sigma)\;d\sigma
\quad\hbox{and}\quad\mathcal{G}(s)=\int_{0}^{s}g^{\ast}(\sigma)\;d\sigma
\qquad\forall s\in\lbrack0,|\Omega|],
\]
relation \eqref{E:4} gives
\[
\label{E:6}1\leq\frac{\left(  -\mu_{w}^{\prime}(\tau)\right)  ^{\frac{\bar{p}%
}{\bar{p}^{\prime}}}}{\Lambda\left(  N\omega_{N}^{\frac{1}{N}}\mu_{w}%
(\tau)^{\frac{1}{N^{\prime}}}\right)  ^{\bar{p}}}\left[  \mathcal{W}(\mu
_{w}(\tau)+\mathcal{G}(\mu_{w}(\tau))\right]  \qquad
\hbox{for $ a.e.$ $\tau >0$},
\]
namely,
\begin{equation}
1\leq\frac{-\mu_{w}^{\prime}(\tau)\Lambda^{-\frac{1}{\bar{p}-1}}}{\left(
N\omega_{N}^{\frac{1}{N}}\right)  ^{\frac{\bar{p}}{\bar{p}-1}}\left(  \mu
_{w}(\tau)\right)  ^{\frac{\bar{p}{\prime}}{N^{\prime}}}}\left[
\mathcal{W}(\mu_{w}(\tau))+\mathcal{G}(\mu_{w}(\tau))\right]  ^{\frac{1}%
{\bar{p}-1}}\qquad\hbox{for $ a.e.$ $\tau >0$}. \label{E:7}%
\end{equation}
Integrating equation \eqref{E:7} between $0$ and $\tau$, we have that
\begin{equation}
\tau\leq\left(  N\omega_{N}^{\frac{1}{N}}\right)  ^{-\bar{p}^{\prime}}%
\Lambda^{-\frac{1}{\bar{p}-1}}\int_{\mu_{w}(\tau)}^{|\Omega|}\sigma
^{-\frac{\bar{p}^{\prime}}{N^{\prime}}}\left[  \mathcal{W}(\sigma
)+\mathcal{G}(\sigma)\right]  ^{\frac{1}{\bar{p}-1}}\;d\sigma\qquad\hbox{for
$\tau>0$}, \label{E:8}%
\end{equation}
and so
\begin{equation}
w^{\ast}(s)\leq\left(  N\omega_{N}^{\frac{1}{N}}\right)  ^{-\bar{p}^{\prime}%
}\Lambda^{-\frac{1}{\bar{p}-1}}\int_{s}^{|\Omega|}\sigma^{-\frac{\bar
{p}^{\prime}}{N^{\prime}}}\left[  \mathcal{W}(\sigma)+\mathcal{G}%
(\sigma)\right]  ^{\frac{1}{\bar{p}-1}}\;d\sigma\qquad
\hbox{for $s\in[0, |\Omega|]$}. \label{E:9}%
\end{equation}
\newline Deriving \eqref{E:9}, we have that
\begin{equation}
\left(  -w^{\ast}(s)\right)  ^{\prime}\leq\left(  N\omega_{N}^{\frac{1}{N}%
}\right)  ^{-\bar{p}^{\prime}}\Lambda^{-\frac{1}{\bar{p}-1}}s^{-\frac{\bar
{p}^{\prime}}{N^{\prime}}}\left[  \mathcal{W}(s)+\mathcal{G}(s)\right]
^{\frac{1}{\bar{p}-1}}\qquad\hbox{for $ a.e. $ $s\in[0,
|\Omega|]$}. \label{E:10}%
\end{equation}
Now let us consider problem (\ref{prob_sym_zero}). We recall that the solution
$z$ of (\ref{prob_sym_zero}) is unique and the symmetry of data assures that
$z(x)=z(\left\vert x\right\vert ),$ \textit{i.e.} $z$ is positive and radially
symmetric. Moreover, putting $s=\omega_{N}\left\vert x\right\vert ^{N}$ and
$\varkappa\left(  s\right)  =z\left(  \left(  s/\omega_{N}\right)
^{1/N}\right)  $ we get for all $s\in\lbrack0,\left\vert \Omega\right\vert ]$%
\[
-\Lambda\left\vert \varkappa\left(  s\right)  \right\vert ^{\bar{p}%
-2}\varkappa^{\prime}\left(  s\right)  =\frac{s^{\bar{p}/N^{\prime}}}{\left(
N\omega_{N}^{1\not / N}\right)  ^{\bar{p}}}\int_{0}^{s}\left(  \lambda
\varkappa^{\ast}(\sigma)+g^{\ast}(\sigma)\right)  \;d\sigma.
\]
It is possible to show (see Lemma 3.2 of \cite{FM}) that the above integral is
positive and this assure that $z(x)=z^{^{\bigstar}}(x).$ By the property of
$z$ we can repeat arguments used to prove (\ref{E:10}), replacing all
inequalities by equalities, we obtain
\begin{equation}
\left(  -z^{\ast}(s)\right)  ^{\prime}=\left(  N\omega_{N}^{\frac{1}{N}%
}\right)  ^{-\bar{p}^{\prime}}\Lambda^{-\frac{1}{\bar{p}-1}}s^{-\frac{\bar
{p}^{\prime}}{N^{\prime}}}\left[  Z(s)+\mathcal{G}(s)\right]  ^{\frac{1}%
{\bar{p}-1}}\qquad\hbox{for $ a.e. $ $s\in[0, |\Omega|]$}, \label{E:11}%
\end{equation}
with
\begin{equation}
Z(s)=\int_{0}^{s}\lambda z^{\ast}(\sigma)\;d\sigma\qquad\forall s\in
\lbrack0,|\Omega|], \label{E:12}%
\end{equation}
where $z$ is the solution to problem \eqref{prob_sym_zero}. From now on, the
proof is a slight modification of the proof of Theorem 5.1 of \cite{DPdB1}
that we recall for the convenience of the reader. We define
\[
H(s)=\int_{0}^{s}\left[  w^{\ast}(\sigma)-z^{\ast}(\sigma)\right]
d\sigma,\quad s\in\lbrack0,|\Omega|]
\]
and we have
\[
H^{^{\prime}}(|\Omega|)=0\text{ and }H(0)=0.
\]
We will show that
\[
H(s)\leq0\qquad\forall s\in\lbrack0,|\Omega|].
\]
We proceed by contradiction and we suppose that there exists $\bar{s}$ such
that
\[
H(\bar{s})=\max_{[0,|\Omega|]}H(s)>0.
\]
\newline We are considering two cases: $\bar{s}=|\Omega|$ and $\bar{s}%
<|\Omega|.$

\textit{i}) If $\bar{s}=|\Omega|$, then there exists $s_{1}$ in $[0,|\Omega|]$
such that
\begin{equation}
H(s_{1})=0\quad\text{and}\quad H(s)>0\qquad\forall s\in(s_{1},|\Omega|].
\label{abs1}%
\end{equation}
Hence, choosing $s$ in $(s_{1},|\Omega|]$ and using \eqref{E:10} and
\eqref{E:11}, we get
\begin{align*}
w^{\ast}(s)  &  =-\int_{s}^{|\Omega|}\frac{d}{d\sigma}w^{\ast}(\sigma
)\,d\sigma\leq\left(  n\omega_{n}^{1/n}\right)  ^{-\bar{p}^{\prime}}\int
_{s}^{|\Omega|}\sigma^{-\frac{\bar{p}^{\prime}}{n^{\prime}}}\left(  \int
_{0}^{\sigma}\left[  g^{\ast}(\tau)-\lambda w^{\ast}(\tau)\right]
d\tau\right)  ^{\frac{\bar{p}^{\prime}}{\bar{p}}}d\sigma\\
&  <\left(  n\omega_{n}^{1/n}\right)  ^{-\bar{p}^{\prime}}\int_{s}^{|\Omega
|}\sigma^{-\frac{\bar{p}^{\prime}}{n^{\prime}}}\left(  \int_{0}^{\sigma
}\left[  g^{\ast}(\tau)-\lambda z^{\ast}(\tau)\right]  d\tau\right)
^{\frac{\bar{p}^{\prime}}{\bar{p}}}d\sigma=-\int_{s}^{|\Omega|}\frac
{d}{d\sigma}z^{\ast}(\sigma)\,d\sigma=z^{\ast}(s)
\end{align*}
in contrast to (\ref{abs1}).

\textit{ii}) If $\bar{s}<|\Omega|$, there exist $s_{1},s_{2}\in\lbrack
0,|\Omega|]$ such that
\begin{equation}
H(s_{1})=0,\quad H(s)>0\quad\hbox{in $(s_1,s_2)$}\qquad\hbox{and
$\quad H^{'}(s_2)\leq 0$}. \label{333}%
\end{equation}
Hence, choosing $s$ in $(s_{1},s_{2})$ and using \eqref{E:10} and
\eqref{E:11}, we obtain
\begin{align*}
w^{\ast}(s)-w^{\ast}(s_{2})  &  =-\int_{s_{2}}^{s}\frac{d}{d\sigma}w^{\ast
}(\sigma)\,d\sigma\leq\left(  n\omega_{n}^{1/n}\right)  ^{-\bar{p}^{\prime}%
}\int_{s_{2}}^{s}\sigma^{-\frac{\bar{p}^{\prime}}{n^{\prime}}}\left(  \int
_{0}^{\sigma}\left[  g^{\ast}(\tau)-\lambda w^{\ast}(\tau)\right]
d\tau\right)  ^{\frac{\bar{p}^{\prime}}{\bar{p}}}d\sigma\\
&  <\left(  n\omega_{n}^{1/n}\right)  ^{-\bar{p}^{\prime}}\int_{s_{2}}%
^{s}\sigma^{-\frac{\bar{p}^{\prime}}{n^{\prime}}}\left(  \int_{0}^{\sigma
}\left[  g^{\ast}(\tau)-\lambda z^{\ast}(\tau)\right]  d\tau\right)
^{\frac{\bar{p}^{\prime}}{\bar{p}}}d\sigma=-\int_{s_{2}}^{s}\frac{d}{d\sigma
}z^{\ast}(\sigma)\,d\sigma=z^{\ast}(s)-z^{\ast}(s_{2}),
\end{align*}
and being $H^{^{\prime}}(s_{2})=w^{\ast}(s_{2})-z^{\ast}(s_{2})\leq0$, we get
\[
w^{\ast}(s)<z^{\ast}(s)\quad\text{in }(s_{1},s_{2})
\]
in contrast to (\ref{333}).
\end{proof}

\bigskip

We are interested in a slight extension of Theorem \ref{ellipticcomp} when the
datum in problem (\ref{prob_sym_zero}) is not the rearrangement of datum $g$
of problem (\ref{problema_zero}), but it is a function that dominates $g.$

\bigskip

\begin{corollary}
\label{corollario elli dominate}Assume the same hypothesis of Theorem
\ref{ellipticcomp}. Let $z$ be the solution to the following problem
\[
\left\{
\begin{array}
[c]{ll}%
-\operatorname{div}(\Lambda|\nabla z|^{\bar{p}-2}\nabla k)+\lambda
z(x)=\widetilde{g}(x) & \hbox{in $\Omega^\bigstar$}\\
& \\
z=0 & \hbox{on $\partial\Omega^\bigstar$},
\end{array}
\right.
\]
where $\widetilde{g}=\widetilde{g}^{\bigstar}$ is a function such that%
\[
\int_{0}^{s}g^{\ast}(\sigma)\text{ }d\sigma\leq\int_{0}^{s}\widetilde{g}%
^{\ast}(\sigma)\text{ }d\sigma\text{ \ \ \ for s}\in\left[  0,\left\vert
\Omega\right\vert \right]  .
\]
Then we have
\[
\int_{0}^{s}w^{\ast}\left(  \sigma\right)  \text{ }d\sigma\leq\int_{0}%
^{s}z^{\ast}\left(  \sigma\right)  \text{ }d\sigma\text{ \ \ \ for}%
\ s\in\left[  0,\left\vert \Omega\right\vert \right]  .
\]

\end{corollary}

\begin{proof}
The result follows reasoning as in the proof of Theorem \ref{ellipticcomp}. In
this case, instead of (\ref{E:7}) we get
\[
1\leq\frac{-\mu_{w}^{\prime}(t)\Lambda^{\frac{1}{\bar{p}-1}}}{\left(
N\omega_{N}^{\frac{1}{N}}\right)  ^{\frac{\bar{p}}{\bar{p}-1}}\left(  \mu
_{w}(t)\right)  ^{\frac{\bar{p}{\prime}}{N^{\prime}}}}\left[  \mathcal{W}%
(\mu_{w}(t))+\mathcal{G}(\mu_{w}(t))\right]  ^{\frac{1}{\bar{p}-1}}%
\qquad\hbox{for $ a.e. \quad t>0$},
\]
with $\mathcal{G}(s)=%
%TCIMACRO{\dint _{0}^{s}}%
%BeginExpansion
{\displaystyle\int_{0}^{s}}
%EndExpansion
\widetilde{g}^{\ast}(\sigma)\;d\sigma.$
\end{proof}

\bigskip

\subsection{Proof of Theorem \ref{Th parabolico}}

Now we are in position to prove Theorem \ref{Th parabolico}. We split the
proof in three steps using the notation introduced after Corollary
\ref{Corollario2}.

%CAMBIARE\ INGLESE/IMPOSTAZIONE (troppo simile al lavoro di Giusy)
\medskip\textit{Step 1. (A priori estimate) }We want to obtain the following a
priori estimate
\begin{equation}
\underset{\lbrack0,T]}{\text{sup}}\int_{\Omega}|u_{M}|^{2}\;dx+\sum_{i=1}%
^{N}\alpha_{i}\int_{0}^{T}\int_{\Omega}|\left(  u_{M}\right)  _{x_{i}}%
|^{p_{i}}\;dx\;dt\leq C, \label{stima}%
\end{equation}
for some constant $C$ depending only on the data. \newline\noindent Let us
consider $u^{m}$ as test function in problem \eqref{discr_pr}. It follows
that
\[
\frac{1}{t_{m+1}-t_{m}}\int_{\Omega}\left(  |u^{m}|^{2}-u^{m}u^{m-1}\right)
\;dx+\int_{\Omega}a^{m}(x,Du^{m})\cdot Du^{m}\;dx=\int_{\Omega}f^{m}%
u^{m}\;dx.
\]
Using (A1), we get
\[
\frac{1}{2}\int_{\Omega}\left(  |u^{m}|^{2}-|u^{m-1}|^{2}+|u^{m}-u^{m-1}%
|^{2}\right)  \;dx+(t_{m+1}-t_{m})\sum_{i=1}^{N}\alpha_{i}\int_{\Omega
}|u_{x_{i}}^{m}|^{p_{i}}dx\leq(t_{m+1}-t_{m})\!\int_{\Omega}|f^{m}%
||u^{m}|\;dx.
\]
Summing on $m$, we obtain
\begin{equation}
\frac{1}{2}\int_{\Omega}|u_{M}(\bar{t},x)|^{2}\;dx+\sum_{i=1}^{N}\alpha
_{i}\int_{0}^{\bar{t}}\int_{\Omega}|\left(  u_{M}\right)  _{x_{i}}|^{p_{i}%
}dxdt\leq\int_{0}^{\bar{t}}\int_{\Omega}|f_{M}||u_{M}|\;dx\;dt+\int_{\Omega
}|u_{0}|^{2}\;dx. \label{sopra}%
\end{equation}
We estimate the right hand side of (\ref{sopra}) using H\"{o}lder inequality
and Young inequality
\begin{align*}
\int_{0}^{\bar{t}}\int_{\Omega}|f_{M}||u_{M}|\;dx\;dt  &  \leq\int_{0}%
^{\bar{t}}\Vert f_{M}\Vert_{\left(  W_{0}^{1,\overrightarrow{p}}%
(\Omega)\right)  ^{\prime}}\Vert u_{M}\Vert_{W_{0}^{1,\overrightarrow{p}%
}(\Omega)}\;dt\leq C\sum_{i=1}^{N}\int_{0}^{\bar{t}}\Vert f_{M}\Vert_{\left(
W_{0}^{1,\overrightarrow{p}}(\Omega)\right)  ^{\prime}}\Vert\,\alpha
_{i}|\left(  u_{M}\right)  _{x_{i}}|\,\Vert_{L^{p_{i}}(\Omega)}\;dt\\
&  \leq C(\varepsilon)\sum_{i=1}^{N}\int_{0}^{\bar{t}}\Vert f_{M}%
\Vert_{\left(  W_{0}^{1,\overrightarrow{p}}(\Omega)\right)  ^{\prime}}%
^{p_{i}^{\prime}}dt+\varepsilon\sum_{i=1}^{N}\alpha_{i}\int_{0}^{\bar{t}}%
\int_{\Omega}|\left(  u_{M}\right)  _{x_{i}}|^{p_{i}}\;dx\;dt.
\end{align*}
Taking $\varepsilon$ small enough and the supremum on $t$, we get (\ref{stima}).

\textit{Step 2}. We prove that
\begin{equation}
\int_{0}^{s}(u^{m})^{\ast}(\sigma,t)\;d\sigma\leq\int_{0}^{s}(v^{m})^{\ast
}(\sigma,t)\;d\sigma, \label{concdiscr}%
\end{equation}
where $u^{m}$ and $v^{m}$ are the solutions of problems (\ref{discr_pr}) and
(\ref{discr_symm_Pr}), respectively.

\noindent We proceed by induction on $m$.

\noindent For $m=1$, by Lemma \ref{riordsum} we have that
\[
\int_{0}^{s}\left(  f^{1}+\frac{u^{0}}{t_{1}-t_{0}}\right)  ^{\ast}%
(\sigma)\;d\sigma\leq\int_{0}^{s}(f^{1})^{\ast}(\sigma)\;d\sigma+\int_{0}%
^{s}\frac{(u^{0})^{\ast}(\sigma)}{t_{1}-t_{0}}\;d\sigma,
\]
and then, by Corollary \ref{corollario elli dominate} it follows that
\[
\int_{0}^{s}\left(  u^{1}\right)  ^{\ast}(\sigma)\;d\sigma\leq\int_{0}%
^{s}\left(  v^{1}\right)  ^{\ast}(\sigma)\;d\sigma,\quad s\in\lbrack
0,|\Omega|].
\]
Now assuming (\ref{concdiscr}) to hold for $m=\nu-1$, we will prove it for
$m=\nu$.

\noindent Using Lemma \ref{riordsum} and the induction hypothesis, we get
\begin{align*}
\int_{0}^{s}\left(  f^{\nu}\!+\frac{u^{\nu-1}}{t_{\nu-1}-t_{\nu-2}}\right)
^{\ast}\!\!(\sigma)d\sigma\!  &  \leq\!\!\int_{0}^{s}\left(  f^{\nu}\right)
^{\ast}(\sigma)\text{ }d\sigma\!+\int_{0}^{s}\frac{\left(  u^{\nu-1}\right)
^{\ast}(\sigma)}{t_{\nu-1}-t_{\nu-2}}\text{ }d\sigma\!\\
&  \leq\!\int_{0}^{s}\left(  f^{\nu}\right)  ^{\ast}(\sigma)\text{ }%
d\sigma\!+\int_{0}^{s}\frac{\left(  v^{\nu-1}\right)  ^{\ast}(\sigma)}%
{t_{\nu-1}-t_{\nu-2}}\text{ }d\sigma.
\end{align*}
So applying Corollary \ref{corollario elli dominate} we get (\ref{concdiscr})
for $m=\nu$.

\textit{Step 3 (Passing to the limit)}. Inequality (\ref{concdiscr}) can be
written as
\begin{equation}
\int_{0}^{s}\left(  u_{M}\right)  ^{\ast}(\sigma,t)\;d\sigma\leq\int_{0}%
^{s}\left(  v_{M}\right)  ^{\ast}(\sigma,t)\;d\sigma, \label{concdiscr2}%
\end{equation}
for $t\in\lbrack0,T]$, where $u_{M}$ and $v_{M}$ are defined by (\ref{UN}) and
(\ref{VN}), respectively.

To conclude, after extracting a subsequence, the estimates (\ref{stima})
yield
\begin{align*}
u_{M}  &  \rightharpoonup u\quad\text{weakly in }L^{\overrightarrow{p}}\left(
0,T;W_{0}^{1,\overrightarrow{p}}(\Omega)\right)  ,\\
v_{M}  &  \rightharpoonup v\quad\text{weakly in }L^{p}(0,T;W_{0}^{1,p}%
(\Omega^{\bigstar})),\\
u_{M}  &  \overset{\ast}{\rightharpoonup}u\quad\text{weakly}\ast
\mbox{ in }L^{\infty}(0,T;L^{2}(\Omega)),\\
v_{M}  &  \overset{\ast}{\rightharpoonup}v\quad\text{weakly}\ast
\mbox{ in }L^{\infty}(0,T;L^{2}(\Omega^{\bigstar})),
\end{align*}
where $u$ and $v$ are the solutions to problems (\ref{prob_parab}) and
(\ref{symm_Pr}), respectively. Note that the last assertion is a consequence
of classical results contained in \cite{Ls}) thanks to which we are able to
pass to the limit in \eqref{concdiscr2} as $M\to+\infty$ and conclude the proof.

\rightline{$\blacksquare$}

%%%%%%%%%%%%%%%%%%%%%%%%%%%%%%%%%%%%
%We thus consider the following elliptic problem
%\begin{equation}\label{Ell_Pr}
%\left\{
%\begin{array}
%[c]{lll}%
%-\left(  a_{ij}^{(1)}(x)\left\vert u_{x_{i}}^{(1)}\right\vert ^{p_{i}%
%-2}u_{x_{i}}^{(1)}\right)  _{x_{j}}+\displaystyle \frac{u^{(1)}}{t_{2}-t_{1}}=\frac{u^{(0)}%
%}{t_{2}-t_{1}}+f^{(1)}(x) &  & \text{in }\Omega\\
%u^{(1)}=0 &  & \text{on }\partial\Omega,
%\end{array}
%\right.
%\end{equation}
%where $u^{(0)}=u_{0}$. Let us denote by $u^{(1)}$ its  solution
%(the existence and uniqueness follows by \cite{lions}).
%\\
%Now we construct the following problem
%\begin{equation}\label{Ell_Pr}
%\left\{
%\begin{array}
%[c]{lll}%
%- {\rm{div}}\left(\Lambda|\nabla v|^{\bar{p} -2}\nabla v\right)
%+\displaystyle \frac{v^{(1)}}{t_{2}-t_{1}}=\frac{v^{(0)}%
%}{t_{2}-t_{1}}+g^{(1)}(x) &  & \text{in }\Omega ^{\bigstar}\\
%v
%^{(1)}=0 &  & \text{on }\partial\Omega^{\bigstar},
%\end{array}
%\right.
%\end{equation}
%where
%\begin{equation}\label{def_g}
%g^{(1)}(x)=(t_2 - t_1)^{-1}\int_{t_1}^{t_2} g(x, t) \; dt.
%\end{equation}

\paragraph*{Acknowledgements}

This work has been partially supported by GNAMPA of INdAM and "Programma
triennale della Ricerca dell'Universit\`{a} degli Studi di Napoli "Parthenope"
- Sostegno alla ricerca individuale 2015-2017".

\end{document}